\documentclass[12pt]{elsarticle}
\usepackage{mathrsfs}
\usepackage{amssymb}
\usepackage{amsmath}
\usepackage{graphicx}
\usepackage{amsmath,amsfonts}
\newtheorem{theorem}{Theorem}
\newtheorem{lemma}[theorem]{Lemma}
\newtheorem{definition}[theorem]{Definition}
\newtheorem{proposition}[theorem]{Proposition}

\newenvironment{proof}{\noindent\\ \noindent\relax{\sc
     Proof}}{{\samepage\par\nopagebreak
     \vspace{0mm}}}
\newcommand{\be}{\begin{equation}} \newcommand{\ee}{\end{equation}}
\newcommand{\ba}{\begin{align}} \newcommand{\ea}{\end{align}}
\newcommand{\baa}{\begin{align*}} \newcommand{\eaa}{\end{align*}}
\newcommand{\ben}{\begin{enumerate}} \newcommand{\een}{\end{enumerate}}
\newcommand{\bi}{\begin{itemize}} \newcommand{\ei}{\end{itemize}}

\newcounter{example}

\journal{Stochastic Processes and their Applications}

\begin{document}
\begin{frontmatter}

\title{Linear-fractional branching processes with countably many types} 
\author{Serik Sagitov}
\ead{serik@chalmers.se}
\address{Mathematical Sciences, Chalmers University of Technology and University of Gothenburg, Sweden}
\begin{abstract}
We study multi-type Bienaym\'e-Galton-Watson processes with linear-fractional reproduction laws using various analytical tools like contour process, spinal representation, Perron-Frobenius theorem for countable matrices, renewal theory. For this special class of branching processes with countably many types we present a transparent criterion for $R$-positive recurrence  with respect to the type space. This criterion appeals to the Malthusian parameter and the mean age at childbearing of the associated linear-fractional Crump-Mode-Jagers process.
\end{abstract}


\end{frontmatter}

\section{Introduction}
Branching processes is a steadily growing body of mathematical research having applications in various areas, primarily in theoretical population biology \cite{HJV}, \cite{KA}, \cite{AP}. A basic version of branching processes, called the Bienaym\'e-Galton-Watson (BGW) process, describes populations of particles which live one unit of time and at the moment of death give birth to a random number of new particles independently of each other. In the single type setting the consecutive population sizes $\{Z^{(n)}\}_{n\ge0}$ form a Markov chain with the state space $\{0,1,2,\ldots\}$. An important analytical tool for studying branching processes is the probability generating functions.  Given $\phi(s)=\mathbb E(s^{Z^{(1)}})$, the $n$-th generation's size is charaterized by the $n$-fold iteration of $\phi(\cdot)$
\[\phi^{(n)}(s)= \mathbb E(s^{Z^{(n)}})=\phi(\ldots(\phi(s))\ldots).\]
Here and elsewhere in this paper  \textit{we always assume that a branching process starts from a single particle.}

The case of linear-fractional generating functions 
\be
\phi(s)=h_0+{h_1s\over1+m-ms}\label{sit}
\ee
with $h_0\in[0,1],h_1=1-h_0$ and $m>0$ is of special interest as their iterations are again linear-fractional functions allowing for explicit calculations of various entities of importance (see \cite{AN}, p. 7). Such explicit results, although being specific, illuminate the known asymptotic results concerning more general branching processes, and, on the other hand, may bring insight into less investigated aspects of the theory of branching processes.

In the multi-type setting particles still reproduce independently but now the number of offspring may depend on the mother's type. A flexible family of population models is obtained by means of BGW-processes with countably many types (\cite{BL}, \cite{H},  \cite{KA}, \cite{ST}). These are infinitely dimensional Markov chains
\[\mathbf{Z}^{(n)}=(Z_1^{(n)},Z_2^{(n)},\ldots),\ n=0,1,2,\ldots,\]
whose $i$-th component $Z_i^{(n)}$ gives the number of particles of type $i$ existing at time $n$.
In this paper we study the class of such branching processes with the generating functions for vectors $\mathbf{Z}^{(n)}$
all being linear-fractional. As shown in Section \ref{SGLF} a linear-fractional BGW-process with countably many types  is fully specified by a triplet of parameters $(\mathbf H,\mathbf g,m)$, where $\mathbf H=(h_{ij})_{i,j=1}^\infty$ is a sub-stochastic matrix,  $\mathbf g=(g_1,g_2,\ldots)$ is a proper probability distribution, and $m$ is a positive constant. 

For a given triplet $(\mathbf H,\mathbf g,m)$ the particles in the linear-fractional BGW-process have the following reproduction law. A particle of type $i$ has no offspring with probability $h_{i0}=1-\sum_{j\ge1}h_{ij}$. Given that this particle has at least one offspring, the type of its first daughter is $j$ with probability $h_{ij}/(1-h_{i0})$, and the number of subsequent daughters has a geometric distribution with mean $m$. With the exception of the first daughter the types of all other offspring particles follow the same distribution $\mathbf g$ independently of each other and {\it independently of mother's type}. 

The countable matrix of the mean offspring numbers
\[
\mathbf{M}=(m_{ij})_{i,j=1}^\infty, \quad m_{ij}=\mathbb{E}(Z_j^{(1)}|\mathbf{Z}^{(0)}=\mathbf{e}_i),
\]
where $\mathbf e_i=(1_{\{i=1\}},1_{\{i=2\}},\ldots)$, in the linear-fractional case is found as 
\begin{align}\label{mrg}
\mathbf M=\mathbf H+m\mathbf H\mathbf 1^{\rm t} \mathbf g,
\end{align}
where $\mathbf 1^{\rm t}$ is the transpose of the row-vector $\mathbf 1^{\rm t}=(1,1,\ldots)$.
Theorem \ref{P2} in Section \ref{SGLF} states that every vector $\mathbf{Z}^{(n)}$  has a multivariate linear-fractional distribution. The generating function of $\mathbf{Z}^{(n)}$ is explicitly expressed in terms of $(\mathbf M,\mathbf g,m)$. An important step in obtaining this formula, see Section \ref{SPP2}, uses a spinal representation argument making the derivation different from that was used in \cite{JL} for the finite-dimensional case.

Assume that the type of the initial particle has distribution $\mathbf g$. Then  as shown  in Section \ref{Sm}, the total population sizes $Z^{(n)}=\mathbf{Z}^{(n)}\mathbf 1^{\rm t}$ for the linear-fractional BGW-process form a single-type discrete time Crump-Mode-Jagers (CMJ) process \cite{JS}, which we call a {\it linear-fractional CMJ-process}. This CMJ-model is not restrictive about the life length distribution, however, the point process of birth events must follow a very specific pattern: at each age a living individual produces independent and identically distributed (iid) geometric numbers of daughters. Making birth events in the linear-fractional CMJ-process very rare (by choosing the row sums of $\mathbf H$ to be close to zero) and rescaling time accordingly we arrive at a continuous time CMJ-process studied in \cite{AL}. In \cite{AL} special attention is paid to the properties of the so-called contour processes of the corresponding planar genealogical trees. The discrete time counterpart of these contour processes is the subject of Section \ref{Scont}.

Section \ref{Seps} introduces a double classification of the BGW-processes with countably many types based on the Perron-Frobenius theorem for countable matrices. Besides the usual classification into subcritical, critical, and supercritical branching processes in the case of infinitely many types one has to distinguish among $R$-transient, $R$-null recurrent, and $R$-positively  recurrent cases depending on the corresponding property of the mean matrix $\mathbf{M}$.
The main result of this paper, Theorem \ref{the1}, among other statements contains a transparent criterion for $R$-positive recurrence of a linear-fractional BGW-process. In terms of the associated CMJ-process this criterion requires that the corresponding Malthusian parameter is well defined and the mean age at childbearing is finite. Theorem \ref{the1} is proven in Section \ref{Spr} using a renewal theory approach.

In Section \ref{Sf} we present basic asymptotic results for subcritical, critical, and supercritical linear-fractional BGW-processes with countably many types in the positively  recurrent case. 
All our results for the  linear-fractional BGW-processes with parameters $(\mathbf H,\mathbf g,m)$ also apply to the case of finitely many, say $a$, types after putting $Z_i^{(0)}=0$, $g_i=0$ for $i\ge a+1$, and $h_{ij}=0$ for $1\le i\le a<j$. The transient and null recurrent cases will be addressed in a separate paper.

\section{Linear-fractional distributions}\label{SGLF}
We are using the following vector notation: $\mathbf{x}=(x_1,x_2,\ldots)$,  $\mathbf s^{\mathbf x}=s_1^{x_1}s_2^{x_2}\ldots$,
$\mathbf{0}=(0,0,\ldots)$. Let $\mathbf{x}^{\rm t}$ stand for
the transpose of the vector $\mathbf x$, and $\mathbf I$ denote the unit matrix $(1_{\{i=j\}})_{i\ge1,j\ge1}$. We denote by $\mathbb Z^\infty_+$ the set of vectors $\mathbf k$ with non-negative integer-valued components and  finite $k=\mathbf k\mathbf 1^{\rm t}$.
\begin{definition}\label{de}
Let $(h_0,h_1,h_2,\ldots)$ be a probability distribution on $\{0,1,2,\ldots\}$, $(g_1,g_2,\ldots)$ be a probability distribution on $\{1,2,\ldots\}$, and $m$ be a positive constant. Put $\mathbf h=(h_1,h_2,\ldots)$, $\mathbf g=(g_1,g_2,\ldots)$. We say that a random vector $\mathbf Z$ has  a linear-fractional distribution LF$(\mathbf h,\mathbf g,m)$ if
\[\mathbb P(\mathbf Z=\mathbf{0})=h_0,\quad \mathbb P(\mathbf Z=\mathbf k+\mathbf{e}_i)={h_im^{k}\over(1+m)^{k+1}}{k\choose k_1,k_2,\ldots}\mathbf g^{\mathbf k}\]
for all $\mathbf k\in\mathbb Z^\infty_+$, where  $k=\mathbf k\mathbf 1^{\rm t}$.
\end{definition}

The name of the distribution is explained by the linear-fractional form of its multivariate generating function
\[
\mathbb E(\mathbf s^{\mathbf Z})=h_{0}+{\sum_{i=1}^\infty h_{i}s_i\over 1+m-m\sum_{i=1}^\infty g_is_i}
\]
which is an extension of its one-dimensional version \eqref{sit}.
Slightly modifying Theorem 1 from \cite{JL} (devoted to the finite-dimensional case) one can demonstrate that Definition \ref{de} covers all possible linear-fractional probability generating functions. 
A  linear-fractional distribution is a geometric distribution modified at zero. Indeed, if $\mathbf{Z}$  has  distribution LF$(\mathbf h,\mathbf g,m)$, then it  can be represented as 
\[\mathbf{Z}=\mathbf{X}+(\mathbf Y_1+\ldots+\mathbf Y_N)\cdot1_{\{\mathbf{X}\ne\mathbf{0}\}}\]
in terms of mutually independent random entities $(\mathbf{X},N,\mathbf{Y}_1,\mathbf{Y}_2,\ldots)$. Here vectors $\mathbf{X}$ and $\mathbf{Y}_j$ have multivariate Bernoulli distributions
\begin{align*}
\mathbb{P}(\mathbf{X}=\mathbf{0})=h_0,\quad\mathbb{P}(\mathbf{X}=\mathbf{e}_i)=h_i,\quad
\mathbb{P}(\mathbf{Y}_j=\mathbf{e}_i)=g_i, \ i\ge1,\ j\ge1,
\end{align*}
and $N$ is a geometric random variable with distribution 
$$\mathbb{P}(N=k)=m^{k}(1+m)^{-k-1}, \ k\ge0.$$ 
Observe that $\mathbf{Z}$ conditionally on $\mathbf{Z}\ne\mathbf 0$ has a multivariate shifted geometric distribution 
\[
\mathbb{E}\big(\mathbf{s}^{\mathbf Z}|\mathbf{Z}\ne \mathbf{0}\big)={(1-h_{0})^{-1}\sum_{j=1}^\infty h_{j}s_j\over1+m-m\sum_{j=1}^\infty g_{j}s_j}.
\]

\begin{definition}\label{def}
Let $\mathbf H=(h_{ij})_{i,j=1}^\infty$ be a sub-stochastic matrix with rows $\mathbf h_i=(h_{i1},h_{i2},\ldots)$ having non-negative elements such that $h_{i0}:=1-h_{i1}-h_{i2}-\ldots$ take values in $[0,1]$.
Let $\mathbf g=(g_1,g_2,\ldots)$ be a probability distribution on $\{1,2,\ldots\}$, and $m$ be a positive constant. A multi-type BGW-process will be called linear-fractional with parameters $(\mathbf H,\mathbf g,m)$, if for all $i=1,2,\ldots$ particles of type $i$ reproduce according to the LF$(\mathbf h_i,\mathbf g,m)$ distribution. 
\end{definition}

Notice the strong limitation on the reproduction law requiring parameters $(\mathbf g,m)$ to be ignorant of mother's type. This is needed for the generating functions 
$$\phi_i^{(n)}(\mathbf{s})=\mathbb{E}( \mathbf{s}^{\mathbf{Z}^{(n)}}|\mathbf{Z}^{(0)}=\mathbf{e}_i),\ i=1,2,\ldots$$
to be also linear-fractional. It is easy to see that if the denominators in 
\[
\phi_i(\mathbf s)\equiv\phi_i^{(1)}(\mathbf{s})=h_{i0}+{\sum_{j=1}^\infty h_{ij}s_j\over 1+m-m\sum_{j=1}^\infty g_js_j}
\]
were different for different $i$, then the iterations of these generating functions
$\phi_i(\phi_1(\mathbf s),\phi_2(\mathbf s),\ldots)$ would lose the linear-fractional property.
\begin{theorem} \label{P2}
Consider a linear-fractional BGW-process with parameters $(\mathbf{H},\mathbf{g},m)$ starting from a type $i$ particle.  Its $n$-th generation size vector $\mathbf Z^{(n)}$ has a linear-fractional distribution LF$(\mathbf{h}_i^{(n)},\mathbf{g}^{(n)},m^{(n)})$ whose parameters satisfy
\begin{align}
m^{(n)}&=  m\sum_{k=0}^{n-1}\mathbf g\mathbf M^k\mathbf{1}^{\rm t},\label{mn}\\
m^{(n)}  {\bf g}^{(n)}&=  m\mathbf{g}(\mathbf{I}+\mathbf{M}+\dots+\mathbf{M}^{n-1}),\label{tn}\\
\mathbf{H}^{(n)}&=\mathbf{M}^{n}-{m^{(n)}\over1+m^{(n)}}\mathbf{M}^{n}\mathbf{1}^{\rm t}\mathbf{g}^{(n)},\label{rjn}
\end{align}
where $\mathbf{H}^{(n)}$ is the matrix with the rows $(\mathbf{h}_i^{(n)})_{i=1}^\infty$.
\end{theorem}

Multiplying \eqref{rjn} by $\mathbf{1}^{\rm t}$ we obtain
\be\label{nex}
\mathbb{P}(\mathbf{Z}^{(n)}\ne \mathbf{0})=(1+m^{(n)})^{-1}\mathbf{M}^{n}\mathbf{1}^{\rm t},
\ee
where $\mathbb{P}(\mathbf{Z}^{(n)}\ne \mathbf{0})$ is a column vector  with elements $\mathbb{P}(\mathbf{Z}^{(n)}\ne \mathbf{0}|\mathbf{Z}^{(0)}=\mathbf{e}_i)$.
Furthermore, Theorem \ref{P2} entails that conditionally on non-extinction $\mathbf{Z}^{(n)}$ has a multivariate shifted geometric distribution 
\begin{align}
\mathbb{E}[\mathbf{s}^{\mathbf{Z}^{(n)}}|\mathbf{Z}^{(n)}\ne \mathbf{0},\mathbf{Z}^{(0)}=\mathbf{e}_i]={(1-h^{(n)}_{i0})^{-1}\sum_{j=1}^\infty h^{(n)}_{ij}s_j\over1+m^{(n)}-m^{(n)}\sum_{j=1}^\infty g^{(n)}_{j}s_j}.
\label{cor}
\end{align}

\section{The linear fractional CMJ-process}\label{Sm}
A  discrete time single type CMJ-process $\{Z^{(n)}\}_{n=0}^\infty$  describes stochastic changes in the population size for a reproduction model  with overlapping generations \cite{JS}. Compared to BGW-populations consisting of {\it particles} it is more appropriate to speak of {\it individuals} building a CMJ-population. Individuals are assumed to live and reproduce independently according to a common life law specifying a reproduction point process $(N_1,\ldots,N_L)$, where $L$ is the life length of an individual and $N_i$ is the number of its daughters produced at age $i$.
In this section we introduce a special class of such processes called  linear-fractional CMJ-processes and discuss its close connection to the class of linear-fractional multi-type BGW-processes.
\begin{definition}
 A  discrete time single type CMJ-process is called linear-fractional, if its individual life law satisfies the following property: $N_L=0$ and given $L=k$ the random variables $N_1,\ldots,N_{k-1}$ are independent and have a common geometric distribution.
\end{definition}
Given $\mathbb P(L>n)=d_n$ and $m=\mathbb E(N_1|L>1)$ a linear-fractional CMJ-process is fully characterized by a pair $(\mathbf d,m)$: the  vector $\mathbf d=(d_1,d_2,\ldots)$ with the non-negative components  satisfying $1\ge d_1\ge d_2\ge\ldots$ and the positive constant $m$.  In particular,  the total offspring number $N_1+\ldots+N_{L-1}$
 has  mean
$\mu=m(\lambda-1)$, where $\lambda=\mathbb E(L)$.

\begin{definition}\label{der}
Consider a linear-fractional CMJ-process with parameters $(\mathbf d,m)$. Let $f(s)=\sum_{n\ge1} d_ns^n$ and put
$$R_f=\inf\{s>0:f(s)=\infty\}.$$ 
If $f(R_f)\ge1/m$, we define the Malthusian parameter $\alpha$ of the CMJ-process as the unique real solution of the equation 
$mf(e^{-\alpha})=1$. If $f(R_f)<1/m$, we put $\alpha=-\infty$. 
\end{definition}

From $\lambda=1+f(1)$  we find $\mu=mf(1)$ and it easy to see that conditions $\mu<1,\mu=1,\mu>1$ are equivalent to $\alpha<0,\alpha=0,\alpha>0$.
In the framework of CMJ-processes \cite{J} a branching process is called subcritical if the Malthusian parameter  is negative $\alpha<0$, critical, if $\alpha=0$, or supercritical, if $\alpha>0$. We point out that given $\alpha>-\infty$
 \[ mf(se^{-\alpha})=\sum_{n=1}^\infty \hat d_ns^n, \ \hat d_n=md_ne^{-\alpha n},\]
  is the generating function for the so-called regeneration age of the immortal individual \cite{JS}. (Notice that in the critical case we have $\hat d_n={\mathbb P(L>n)\over\mathbb E(L)-1}$.)
The corresponding mean value
  \[
\beta=m\sum_{n=1}^\infty n d_ne^{-\alpha n}
\]
is usually called the mean age at childbearing \cite{J}. For $\alpha=-\infty$ we put $\beta=\infty$. 

\vspace{3mm}
\noindent{\bf Example 1}. 
Let $d_n=c_nn^{-k}e^{-\gamma n}$ for some constants $\gamma\ge0$, $k\ge0$, and assume $0<\liminf_nc_n<\limsup_nc_n<\infty$. Clearly, in this case $R_f=e^\gamma$ and putting $A=\sum_{n=1}^\infty c_nn^{-k}$ we get
\begin{itemize}
\item  if $A>1/m$, then $-\gamma<\alpha<\infty$ and $\beta<\infty$,
\item  if $A=1/m$, then  $\alpha=-\gamma$ and $\beta<\infty$ iff  $k>2$,
\item  if $A<1/m$, then $\alpha=-\infty$ and $\beta=\infty$.
\end{itemize}
\vspace{3mm}

It turns out that for a given triplet  $({\mathbf{H}},{\mathbf g},m)$ the corresponding linear-fractional BGW-process  $\mathbf Z^{(n)}$ can be associated with a linear-fractional CMJ-process  $Z^{(n)}=\mathbf Z^{(n)}\mathbf 1^{\rm t}$ characterized by a pair $(\mathbf d,m)$, where the vector  $\mathbf d$ has components
\be\label{qn}
d_n=\mathbf{g}\mathbf H^{n}\mathbf{1}^{\rm t},\ n\ge1.
\ee
To see this let  the initial particle of the BGW-process have distribution $\mathbf g$.
If the initial particle dies without producing any offspring we say that the initial individual in the associated CMJ-process had the life length $L=1$ and, if the initial particle produced at least one offspring we say $L>1$. Obviously, $\mathbb P(L>1)=\mathbf{g}\mathbf H\mathbf{1}^{\rm t}$. The key idea of defining the associated CMJ-process is to view an individual as a sequence of first-born descendants of a particle which itself is either the progenitor particle or a particle which is not first-born. Following this idea  given $L>1$, we say $L>2$ if the first-born daughter of the progenitor produces at least one offspring particle. This explains \eqref{qn} for $n=2$. Continuing in the same manner we see that  \eqref{qn} indeed gives us the distribution of the  life length of the progenitor individual. The fact that all daughter individuals behave in the way prescribed by the linear-fractional CMJ-model is a straightforward consequence of the particle properties of the linear-fractional BGW-process.

Turn for visual help to Figure \ref{f1}.C which gives the individual based picture of the same genealogical tree as in Figure \ref{f1}.A.  Each vertical arrowed branch in Figure \ref{f1}.C represents an individual dying before the observation time $n=5$. In particular, the initial individual lives two units of time producing two daughters: one of them lives two units of time and the other only one. We see also  that the first granddaughter of the initial individual produces two daughters at different ages.

The associated CMJ-process $Z^{(n)}$ tracks only the total number of BGW-particles at time $n$ ignoring the information on the types of the particles. To recover this information we may introduce additional labeling of individuals using the types of underlying BGW-particles. The evolution of a labeled individual over the type space can be modeled by a Markov chain whose state space $\{0,1,2,\ldots\}$ is the type space  $\{1,2,\ldots\}$ of the BGW-process augmented with a graveyard state  $\{0\}$. The transition probabilities of such a chain are given by a stochastic matrix $\mathring{\mathbf{H}}=(\mathring h_{ij})_{i,j\ge0}$ 
 with $\mathring h_{ij}=h_{ij}$ for $i\ge1$, $j\ge0$, $\mathring h_{0 j}=0$ for $j\ge1$, and $\mathring h_{00}=1$. In terms of this Markov chain the life length $L$ is the time until absorption at $\{0\}$ starting from a state  $j\in\{1,2,\ldots\}$  with probability  $g_j$. (Distributions describing absorption times in Markov chains are called {\it phase-type distributions}, see for example \cite{A}.) A labeled individual is able to visit all elements of the type space except phantom types defined next.
\begin{definition}\label{pha}
Consider a linear-fractional multi-type BGW-process with parameters  $({\mathbf{H}},{\mathbf g},m)$. If the $j$-th element of the vector $\mathbf{g}\mathbf H^{n}$ is zero for all $n\ge0$, we call $j$ a phantom type of this BGW-process.
\end{definition}

Among many pairs $({\mathbf{H}},{\mathbf g})$ resulting in the same life length distribution vector $\mathbf d$ it is worth to emphasize the next one
\[\mathbf{H}=
\left(
\begin{array}{cccccc}
0&d_1&0&0&0&\ldots\\
0&0&d_2/d_1&0&0&\ldots\\
0&0&0&d_3/d_2&0&\ldots\\
0&0&0&0&d_4/d_3&\ldots\\
\ldots&\ldots&\ldots&\ldots&\ldots&\ldots
\end{array}\right),\ \ \ \mathbf g=\mathbf e_1.
\]
For this particular choice of $({\mathbf{H}},{\mathbf g})$ it easy to verify that \eqref{qn} holds. In this case the particle type can be viewed as the age of the corresponding individual. Obviously,  there are no phantom types if $d_n>0$ for all $n\ge1$. If $d_a>0$ and $d_{a+1}=0$ for some natural $a$, then to avoid the phantom types we must restrict the state space to $\{1,\ldots,a\}$.
Observe that with $\mathbf g=\mathbf e_2$ type $j=1$ becomes a phantom type.

\begin{figure}
\centering
\includegraphics[height=4cm, width=13.8cm]{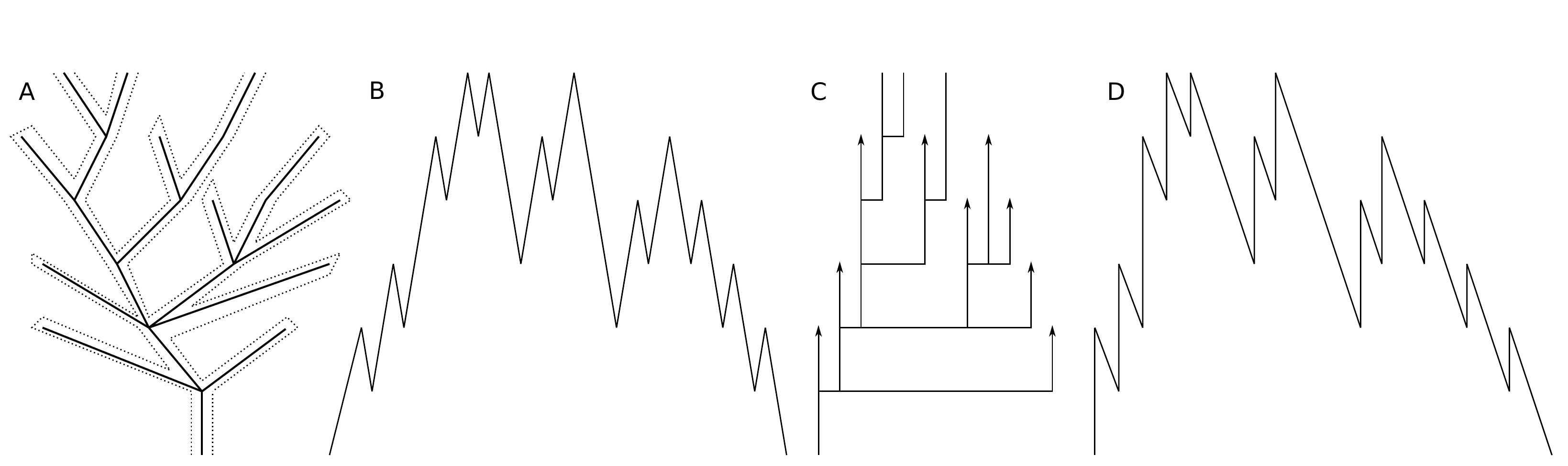}
\caption{{\bf A}: a BGW-tree (thick line) stopped at level $n=5$ and supplied with a contour  (dashed)  line. {\bf B}: the corresponding unfolded contour process.  {\bf C}: the CMJ-view of the same tree when depicted in terms of individuals forming a branching process with overlapping generations. The vertices marked by arrows represent individuals which are dead at that time. The stopped tree gives no information about the fate of the three tip vertices. {\bf D}: the modified contour process of a constant speed descent with iid upward jumps. }
\label{f1}
\end{figure}

\section{Jumping contour processes in discrete time}\label{Scont}
In this section we remind the concept of a contour process generated by a planar BGW-tree. Then we show that in the multivariate linear-fractional framework a jumping version of the contour process (in the spirit of  \cite{AL}) has a nice Markovian structure of a constant speed descent with independent and identically distributed upward jumps. In the end of this section we apply the contour process method to give a short proof of the first statement in  Theorem \ref{P2}.

Both the contour process and the spinal representation methods (the latter discussed in Section \ref{SPP2}) rely on a planar genealogical tree connecting the particles of the branching processes appeared up to the time of observation. For the current setting of linear-fractional BGW-processes it is important to use a particular planar version of the genealogical tree: given a group of siblings stemming from the same particle, the {\it leftmost branch} should connect the mother to its {\it first daughter} (that one whose type may depend on mother's type).

For a given planar tree, its contour profile is defined by the depth-first search procedure. Figure \ref{f1} illustrates the basic definition of the contour process for a finite tree supplied with a path around the tree. The contour process is simply the seesaw line graph (panel B)  representing the height of the location of a virtual car driving with a constant speed along the path outlined on the panel A. Notice that the $x$-axis in the panel A is introduced just to distinguish among different branches on the same level, hence the speed of the car is meant along the $y$-axis.
The resulting contour B of the tree A is an excursion of a random walk starting and ending at level $-1$. 
Even if the realization of the genealogical tree is infinite, one can still work with the contour processes after cutting off the branches above level $n$ corresponding to the observation time, as shown in Figure \ref{f1}. 

It is easy to reconstruct the tree on the panel A from the contour process on the panel B. As we said, the whole tree A is represented by the excursion B 
starting  and ending at the bottom level $-1$. Raising the bottom level from $-1$ to $0$ will split the tree A into 3 subtrees stemming from 3 daughters of the progenitor particle. At the same time the excursion B becomes split into 3 sub-excursions starting  and ending at level $0$. Proceeding in this way by moving up the bottom level and observing how the excursions are decomposed in sub-excursions  allows us to fully reconstruct the branching history of the original genealogical tree.

The contour process approach has proven to be very useful in the theory of branching processes  (see for example \cite{GK} and references therein). In the single type linear-fractional case \eqref{sit} the contour process has a simple structure of an alternating random walk.  Alternating upward and downward stretches have independent lengths following shifted geometric laws having mean $h_0^{-1}$ for the upward stretches and mean ${1+m\over m}$ for the other for downward stretches.

In the multi-type linear-fractional setting, one can ensure a Markov property of the contour process by introducing additional labeling of the vertices in the contour path. Each vertex will be labeled by a pair of integers $(l,i)$ with $l\ge-1$ and $i\ge0$. The current state $(l,i)$ with $l\ge0,i\ge1$ tells three things about the contour process: the current level is $l$, the last move was upward, and the underlying BGW-particle is of type $i$. If the contour process is at the vertex labeled $(l,0)$, then again its current level is $l$ but now we know that this level was  attained after a downward step. 
Such a labeled contour process (if we ignore the first compulsory move from level $-1$ to 0) can be viewed as a Markov chain with the following  transition probabilities
\begin{align*}
&\mathbb{P}\{(l,i)\to(l+1,j)\}=h_{ij}, \qquad \qquad\mathbb{P}\{(l,i)\to(l-1,0)\}=h_{i0},\\
&\mathbb{P}\{(l,0)\to(l+1,j)\}={m\over1+m}g_j, \quad \ \mathbb{P}\{(l,0)\to(l-1,0)\}={1\over1+m},\\
&\mathbb{P}\{(-1,0)\to(-1,0)\}=1,
\end{align*}
for all $i\ge1,\ j\ge1,\ l\ge0$.

The following alternative way of introducing Markovian structure in the contour process of a linear-fractional multi-type branching process does not require additional labeling. What we call here the {\it jumping contour process} (cf  \cite{AL}) has a trajectory of a constant speed descent with independent upward jumps each distributed as the individual life length $L$. The process starts from level $-1$ with an instantaneous jump and proceeds as follows.
From any given current level $l$ the jumping contour process moves one level down to $l-1$ and either settles there with probability ${1\over1+m}$ or, with probability ${m\over1+m}$, it instantaneously jumps say $k$ levels up coming to the level $k+l-1$. Figure \ref{f1}.D clearly illustrates the last construction.

Notice that  super-criticality of the underlying branching process can be identified via a positive drift for the contour process.
The drift $\lambda-1-m^{-1}$ of the jumping contour process is computed as the difference between the mean jump size $\lambda=\mathbb E(L)$ and the average length of a downward stretch $1+m^{-1}$. Clearly, the inequality $\lambda-1-m^{-1}>0$ is equivalent to $\mu>1$.

\vspace{3mm} 
\noindent{\sc Proof of Theorem \ref{P2}, part 1.} Consider a linear-fractional BGW-process with parameters $(\mathbf{H},\mathbf{g},m)$. Next we show that if the BGW-process stems from a particle of type $i$, then its vector of $n$-th generation sizes  $\mathbf Z^{(n)}$ has a linear-fractional joint distribution with some  parameters $(\mathbf{h}_i^{(n)},\mathbf{g}^{(n)},m^{(n)})$ so that $(\mathbf{g}^{(n)},m^{(n)})$ are independent of $i$. Consider the genealogical tree of the linear-fractional BGW-process stopped at level $n$ and denote by $j$ the type of the leftmost tip in the tree if any. The random vector  $\mathbf Z^{(n)}$ counts the tips of various types and  $Z^{(n)}=\mathbf Z^{(n)}\mathbf{1}^{\rm t}$ gives the total number of the tree tips. We have to verify that conditioned on $Z^{(n)}\ge1$ the following two properties hold:
\begin{enumerate}
\item the number $Z^{(n)}-1$ of the tree tips  to the right of the leftmost tip has a geometric distribution which is independent from the types $(i,j)$,
\item the types of these $Z^{(n)}-1$ tips are iid and independent from  $(i,j)$.
\end{enumerate}
These properties are simple consequences of the following Markov features of the contour process described above:
\begin{itemize}
\item  the number of particles alive at time $n$, if any, is 1 plus the number of excursions of the contour process starting at level $n$ downwards and coming back to the level $n$ escaping absorption at level $-1$,
\item the future of the contour process that just made a downward move depends only on the current level and has no memory of the earlier path.
\end{itemize}
It follows that $Z^{(n)}-1$ has a geometric  distribution whose parameter is the probability that the jumping contour process starting downwards from level $n$ will be absorbed at level $-1$ without visiting level $n$ once again.

\section{Classification of branching processes with countably many types}\label{Seps}

Multi-type BGW-processes are classified according to the asymptotic properties of the mean matrices $\mathbf{M}^{(n)}=(m_{ij}^{(n)})_{i,j=1}^\infty$ with elements
$$m_{ij}^{(n)}=\mathbb{E}(Z_j^{(n)}|\mathbf{Z}^{(0)}=\mathbf{e}_i)$$
as $n\to\infty$. The assumed independence of particles implies a recursion $\mathbf{M}^{(n)}=\mathbf{M}\mathbf{M}^{(n-1)}$, where $\mathbf{M}=\mathbf{M}^{(1)}$. It follows that  $\mathbf{M}^{(n)}=\mathbf{M}^n.$
Given that all powers $\mathbf{M}^n$ are element-wise finite (which is always true in the linear-fractional case) the asymptotic behavior of these powers is  described by the Perron-Frobenius theory for countable matrices (see Chapter 6 in \cite{ES}).  

Next we remind some crucial conclusions from this theory holding for an {\it irreducible and aperiodic} countable matrix $\mathbf{M}$.
Recall that a non-negative matrix $\mathbf M$ is called irreducible, if for any pair of indices $(i,j)$ there is a natural number  $n$ such that $m_{ij}^{(n)}>0$. The period of an index $i$ in an irreducible matrix $\mathbf M$ is defined as the greatest common divisor of all natural numbers  $n$ such that $m_{ij}^{(n)}>0$. In the irreducible case all such indices have the same period which is called the period of $\mathbf M$. When this period equals one the matrix $\mathbf M$ is called aperiodic. 

Due to Theorem 6.1 from \cite{ES} all elements of the matrix power series
$\mathbf M(s)=\sum_{n\ge0}s^n\mathbf M^n$
have a common convergence radius $0\le R<\infty$, called the convergence parameter of the matrix $\mathbf{M}$. Furthermore, one of the two alternatives holds:
\begin{itemize}
\item $R\mbox{-transient case: }\sum_{n=0}^\infty m_{ii}^{(n)}R^n<\infty, \ i\ge1,$
\item $R\mbox{-recurrent case: }\sum_{n=0}^\infty m_{ii}^{(n)}R^n=\infty, \ i\ge1.$
\end{itemize}
%
According to \cite{ES}  (Theorem 6.2 and a remark afterwards) in the $R$-recurrent case there exist unique up to constant multipliers {\it positive} vectors $\mathbf{u}$ and $\mathbf{v}$ such that
\[R\mathbf{M}\mathbf{u}^{\rm t}=\mathbf{u}^{\rm t}, \ R\mathbf{v}\mathbf{M}=\mathbf{v}.\]
Using $Rv_jm_{ji}/v_i$ one can transform the matrix $\mathbf{M}$ into a stochastic matrix.

The $R$-recurrent case is further divided in two sub-cases: $R$-null, when  $\mathbf{v}\mathbf{u}^{\rm t}=\infty$, and $R$-positive with $\mathbf{v}\mathbf{u}^{\rm t}<\infty$.
In the $R$-null case (and clearly also in the $R$-transient case)
\be\label{nu}
R^nm_{ij}^{(n)}\to0 \mbox{ for all } i,j\ge1.
\ee
In the $R$-positive case (Theorem 6.5 from \cite{ES}) one can scale the eigenvectors so that $\mathbf{v}\mathbf{u}^{\rm t}=1$ and obtain
\be\label{posi}
R^nm_{ij}^{(n)}\to u_iv_j \mbox{ for all } i,j\ge1.
\ee

These results suggest a double classification of the BGW-processes with countably many types having a mean matrix $\mathbf{M}$. The usual classification of the multi-type BGW-processes satisfying relation \eqref{posi} depends on the Perron-Frobenius eigenvalue  $\rho=1/R$. Given $\rho<1$, $\rho=1$, or $\rho>1$ the branching process is called subcritical, critical, or supercritical. In view of possibilities other than  \eqref{posi} an additional classification is needed to account for particles escaping to infinity across the type space.

\begin{definition}
A BGW-process with countably many types will be called subcritical (critical, supercritical) and transient \{recurrent, null-recurrent, positively recurrent\} in the type space, if  its matrix of the mean offspring numbers $\mathbf{M}$ has a convergence radius $R>1$ ($R=1$, $R<1$) and is $R$-transient \{$R$-recurrent, $R$-null recurrent, $R$-positively  recurrent\}.
\end{definition}

There are several published results for the BGW-processes with countably many types (see for example \cite{AK}, \cite{K}). One of them is Theorem 1 in \cite{M} dealing with the $R$-positively  recurrent  supercritical $(R<1)$ case. It states that if 
$$\sum_{i=1}^\infty v_i \mathbb E\left((\mathbf{Z}^{(1)}\mathbf{u}^{\rm t})^2|\mathbf{Z}^{(0)}=\mathbf{e}_i\right)<\infty,$$ 
then for any $\mathbf{w}$ such that $\mathbf{w}\le c\mathbf{u}$ for some positive constant $c$, the convergence 
$R^n\mathbf{Z}_n\mathbf{w}^{\rm t}\to Y\mathbf{v}\mathbf{w}^{\rm t}$
holds in mean square, where $Y\ge0$ has a finite second moment. 
This statement is cited here just to illustrate the need for finding illuminating examples of branching processes, where conditions like $R$-positive recurrence could be verified and the values of $(R,\mathbf u, \mathbf v)$ be computed in terms of the basic model parameters. 

Returning a linear-fractional branching process with parameters  $(\mathbf{H},\mathbf{g},m)$ consider the matrix $\mathbf M$ of the mean offspring numbers given by \eqref{mrg}. Clearly, irreducibility of  $\mathbf M$  prohibits the phantom types, see Difinition \ref{pha}.  The opposite is not true, if there exist so-called final types that never produce offspring, in other words, if $\mathbf H$ contains zero rows.

%

\begin{theorem}\label{the1}
The matrix  $\mathbf M$ given by \eqref{mrg} is irreducible if and only if  there are no phantom types and $\mathbf H$ does not contain zero rows. 
If  $\mathbf M$ is irreducible and aperiodic, the following three statements are valid
\begin{quote}
\begin{description}
\item[ (i)]  the convergence parameter of $\mathbf M$ is computed  as 
 \be\label{star}
 R=
\left\{
\begin{array}{ll}
e^{-\alpha},  & \mbox{ if }  \alpha>-\infty,   \\
R_f,  &  \mbox{ if }  \alpha=-\infty,    
\end{array}
\right.
\ee
using Definition \ref{der} and formula \eqref{qn} for the  components of $\mathbf d$,
\item[ (ii)]  $\mathbf M$ is $R$-recurrent if and only if $\alpha>-\infty$,
\item[ (iii)]  $\mathbf M$ is $R$-positively  recurrent if and only if $\beta<\infty$.
\end{description}
\end{quote}
In the latter case
\be\label{mnu}
R^n\mathbf M^n\to \mathbf u^{\rm t}\mathbf v,\ n\to\infty,
\ee
where element-wise positive and finite vectors $\mathbf u$ and $\mathbf v$ are given by
\begin{align}
\mathbf u^{\rm t}&=(1+m)\beta^{-1}\sum_{k=1}^\infty R^k\mathbf H^k\mathbf{1}^{\rm t},\label{muu}\\
\mathbf v&={m\over 1+m}\sum_{k=0}^\infty R^k\mathbf{g}\mathbf H^k\label{muv},
\end{align}
and satisfy $\mathbf v\mathbf u^{\rm t}=\mathbf v\mathbf 1^{\rm t}=1$ as well as  $\mathbf g\mathbf u^{\rm t}={1+m\over m\beta}$.
 
\end{theorem}

\noindent{\bf Example 2}.
Assume that for some positive constant $r$ the pair $(\mathbf H,\mathbf g)$ satisfies one or both of the following conditions 
\begin{enumerate}
\item $\mathbf g\mathbf H=r\mathbf g$ so that $\mathbf g\mathbf M=(1+m)r\mathbf g$ and  $ \mathbf v=\mathbf g$,
\item $\mathbf H\mathbf 1^{\rm t}=r\mathbf 1^{\rm t}$ so that  $\mathbf M\mathbf 1^{\rm t}=(1+m)r\mathbf 1^{\rm t}$ and $\mathbf u=\mathbf 1$ . 
\end{enumerate}
We have necessarily $r\le1$ since
$r=r\mathbf g\mathbf 1^{\rm t}=\mathbf g\mathbf H\mathbf 1^{\rm t}\le\mathbf g1^{\rm t}=1$.
In both cases we obtain $\rho=(1+m)r,\ \beta={1+m\over m}$, and $\mathbb{P}(L>n)=r^n$. Notice that  $\mathbb{P}(L=\infty)=1$ for $r=1$.\\

\section{$R$-positively  recurrent case}\label{Sf}
Consider a linear-fractional BGW-process with an irreducible and aperiodic  $\mathbf{M}$ assuming $\beta<\infty$. In this case according to Theorem \ref{the1} we have $\mathbf M^n\sim \rho^n\mathbf{u}^{\rm t}\mathbf{v}$,  
where $\rho=R^{-1}=e^\alpha$. It follows that the left eigenvector $\mathbf{v}$ describes the stable type distribution: $\mathbf e_i\mathbf M^n\sim u_i\rho^n\mathbf v$,
and the right eigenvector $\mathbf{u}$ compares productivity of different types: $\mathbf M^n\mathbf 1^{\rm t}\sim \rho^n\mathbf u^{\rm t}$ (so that $u_i$ can be interpreted as the ``reproductive value" of type $i$).
The next three propositions present basic asymptotic results 
for the linear-fractional BGW-processes extending similar statements for the finite-dimensional case obtained in \cite{JL} and \cite{P}. 

\begin{proposition}\label{p1} In the subcritical positively recurrent case when  $\rho<1$, or equivalently $\mu<1$,
\be
\mathbb{P}\big(\mathbf{Z}^{(n)}\ne \mathbf{0}\big)\sim\rho^n(1+m)^{-1}(1-\mu)\mathbf u^{\rm t}.
\label{epr}
\ee
Furthermore, for any initial type $i$ we get
\begin{align*}
\mathbb P\big(\mathbf{Z}^{(n)}=\mathbf k|\mathbf{Z}^{(n)}\ne \mathbf{0},\mathbf{Z}^{(0)}=\mathbf{e}_i\big)\to \mathbb P(\mathbf Y=\mathbf k) \mbox{ for all $\mathbf k\in\mathbb Z^\infty_+$},
\end{align*}
where $\mathbf Y$ has a distribution LF$(\tilde{\mathbf h},\tilde{\mathbf g},\tilde m)$ with $\tilde m=m\lambda(1-\mu)^{-1}$,
\begin{align*}
\tilde{\mathbf h}&=(1+m)(1-\mu)^{-1}\mathbf v-m\mathbf g(\mathbf{I}-\mathbf{M})^{-1},\qquad \tilde{\mathbf h}\mathbf 1^{\rm t}=1,\\ 
\tilde{\mathbf g}&=\lambda^{-1}(1-\mu)\mathbf g(\mathbf{I}-\mathbf{M})^{-1}.
\end{align*}
\end{proposition}

\begin{proposition}\label{p2}  In the critical positively recurrent case when   $\rho=1$ we have
\[\mathbb{P}(\mathbf{Z}^{(n)}\ne \mathbf{0})\sim\ n^{-1}(1+m)^{-1}\beta\mathbf u^{\rm t}.\]
If a vector $\mathbf{w}$ has bounded components $(\sup_{j\ge1}|w_j|<\infty)$ and $\mathbf{v}\mathbf{w}^{\rm t}>0$, then  for all $x>0$ and $i\ge1$
\begin{align*}
\mathbb P\big(\mathbf{Z}^{(n)}\mathbf{w}^{\rm t}>nx|\mathbf{Z}^{(n)}\ne \mathbf{0},\mathbf{Z}^{(0)}=\mathbf{e}_i\big)\to e^{-x/c_w}, \quad c_w=
(1+m)\beta^{-1}\mathbf{v}\mathbf{w}^{\rm t}.
\end{align*}
In other words, conditionally  on non-extinction $n^{-1}\mathbf{Z}^{(n)}$ weakly converges to $X\mathbf{v}$, where  $X$ is exponentially distributed with mean $(1+m)\beta^{-1}$. 
\end{proposition}

\begin{proposition}\label{p3}  In the supercritical positively recurrent case when  $\rho>1$
\[\mathbb{P}(\mathbf{Z}^{(n)}\ne \mathbf{0})\to\ (\rho-1)(1+m)^{-1}\beta\mathbf u^{\rm t}.\]
Furthermore,  for any $\mathbf{w}$ with bounded components and $\mathbf{v}\mathbf{w}^{\rm t}>0$ 
\begin{align*}
\mathbb P\big(\mathbf{Z}^{(n)}\mathbf{w}^{\rm t}>\rho^nx|\mathbf{Z}^{(n)}\ne \mathbf{0},\mathbf{Z}^{(0)}=\mathbf{e}_i\big)\to e^{-x(\rho-1)/c_w},\quad x>0.
\end{align*}
\end{proposition}

As straightforward corollaries of Propositions \ref{p1}, \ref{p2}, \ref{p3} we get the following asymptotic results for the linear-fractional CMJ-processes with $\beta<\infty$. The survival probability 
$\mathbb P(Z^{(n)}>0)=\sum_{i\ge1} g_i\mathbb{P}\big(\mathbf{Z}^{(n)}\ne \mathbf{0}|\mathbf{Z}^{(0)}=\mathbf{e}_i\big)$
satisfies a particularly transparent asymptotical formula
\begin{align*}
 \mathbb P(Z^{(n)}>0)\sim 
\left\{
\begin{array}{ll}
e^{\alpha n}(1-\mu)(m\beta)^{-1},  &  \mbox{ if }\alpha<0,   \\
(nm)^{-1}, &   \mbox{  if }\alpha=0,  \\
 (e^\alpha-1)m^{-1},  &  \mbox{  if }\alpha>0.  
\end{array}
\right.
\end{align*}
Moreover, in the subcritical case we get a geometric conditional limit distribution 
$$\mathbb P(Z^{(n)}=k|Z^{(n)}>0)\to m^{k-1}(1+m)^{-k}, \ k\ge1,$$
in the critical case we have 
$$\mathbb P(Z^{(n)}>nx|Z^{(n)}>0)\to e^{-\beta x/(1+m)}, \quad x>0,$$
and in the supercritical case 
$$\mathbb P(Z^{(n)}>e^{\alpha n}x|Z^{(n)}>0)\to \exp\{{-x\beta (e^\alpha-1)/(1+m)}\}, \quad x>0.$$
These explicit results illuminate much more general limit theorems for the CMJ-processes available in \cite{J}, \cite{JN},  and \cite{S}.

\section{Proofs of  Theorems \ref{P2}, \ref{the1} and Propositions \ref{p1}, \ref{p2}, \ref{p3}}\label{SP}
\subsection{Proof of Theorem \ref{P2}, part 2}\label{SPP2}

In Section \ref{Scont} we have shown that $\mathbf Z^{(n)}$ has a linear-fractional distribution with unspecified parameters $(\mathbf{H}^{(n)},\mathbf{g}^{(n)},m^{(n)})$.
Turning to the proof of relations  \eqref{mn}, \eqref{tn},  and \eqref{rjn} observe first that after multiplying \eqref{mrg} by $\mathbf 1^{\rm t}$ we obtain $\mathbf M\mathbf{1}^{\rm t}=(1+m)\mathbf H\mathbf 1^{\rm t}$,
which leads to a useful reverse expression of $\mathbf{H}$ in terms of $\mathbf{M}$
\[
\mathbf{H}=\mathbf{M}-{m\over1+m}\mathbf{M}\mathbf{1}^{\rm t}\mathbf{g}.
\]
Clearly,  relation \eqref{rjn} is a straightforward counterpart of the last relation applied to the linear-fractional distribution of $\mathbf Z^{(n)}$. 

We prove \eqref{mn} using the spinal representation of the BGW-tree illustrated in Figure \ref{f3}. 
\begin{figure}
\centering
\includegraphics[
height=4cm,width=13cm]{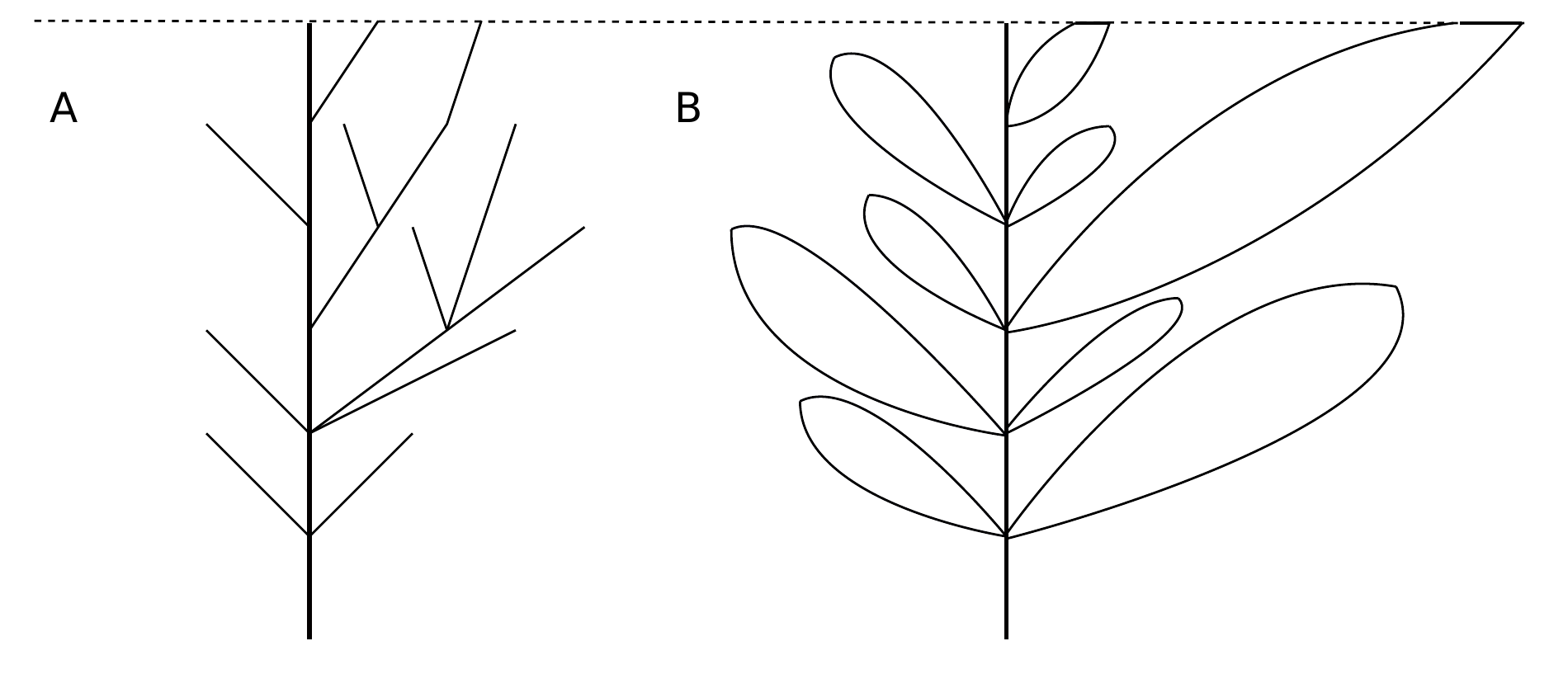}
\caption{A spinal representation of a BGW-tree reaching the observation level. {\bf A}: the spinal representation of the tree from Figure \ref{f1}. {\bf B}: a schematic view of the spinal representation, cf \cite{G}.  }
\label{f3}
\end{figure}
Suppose that $Z^{(n)}\ge1$. The corresponding spine of the planar BGW-tree is the leftmost lineage of particles reaching the level $n$. Recall that the unspecified parameter $m^{(n)}$ is the mean number of all branches present at level $n$ except the spinal one. Since this mean is the sum of contributions from all the lineages stemming  to the right
of the spine (see Figure \ref{f3}), to establish equality \eqref{mn} it suffices to see that the average number of particles stemming from the spinal particle at time $k\in[0,n-1]$ equals $m\mathbf g\mathbf M^{n-k-1}\mathbf{1}^{\rm t}$. The last assertion is a straightforward consequence of the memoryless property of geometric distribution:
\begin{itemize}
\item in the linear-fractional case at each level $k\in[0,n-1]$ there is a geometric with mean $m$ number of branches growing off the spine to the right of it,
\item every one of such daughter branching processes produces on average $\mathbf g\mathbf M^{n-k-1}\mathbf{1}^{\rm t}$ particles at time $n$.
\end{itemize}
Equality \eqref{tn} is obtained using the same argument. It is just a detailed version of  \eqref{mn} taking into account the numbers  of particles of various types existing at time $n$.

%
 
\subsection{Renewal theory argument}\label{Sr}

This section contains two lemmata used in Section \ref{Spr}. The first lemma deals with two power series $\mathbf M(s)=\sum_{n\ge0}s^n\mathbf M^n$ and  $\mathbf H(s)=\sum_{n\ge0} s^n\mathbf H^n$. 
\begin{lemma}\label{mps}
Let $f(s)=\sum_{n\ge1} d_ns^n$ with $d_n$ given by \eqref{qn}. The matrix-valued function  $\mathbf{M}(s)$  is element-wise finite if and only if $mf(s)<1$. In this case
\[
\mathbf{M}(s)=\mathbf{H}(s)+{m\over1-mf(s)}(\mathbf{H}(s)-\mathbf{I})\mathbf{1}^{\rm t}\mathbf{g}\mathbf H(s).
\]
\end{lemma}
\begin{proof}.
Due to  \eqref{mrg}  we have $\mathbf{M}^{n+1}=(\mathbf{H}+m\mathbf{H}\mathbf{G})\mathbf{M}^n$, where $\mathbf{G}=\mathbf{1}^{\rm t}\mathbf g$. Using induction we obtain
\begin{align*}
\mathbf{M}^{n}=\mathbf{H}^n+m\sum_{i=1}^n\mathbf{H}^i\mathbf{G}\mathbf{M}^{n-i}.
\end{align*}
Putting $\bar{\mathbf H}(s)=\mathbf H(s)-\mathbf I$ and  $M(s)=\mathbf g\mathbf M(s)\mathbf{1}^{\rm t}$ we derive first
\begin{align}
\mathbf M(s)&=\mathbf H(s)+m\bar{\mathbf H}(s)\mathbf G\mathbf M(s)
\label{MHG}
\end{align}
and then
\[ M(s)=1+f(s)+mf(s)M(s).\]
Thus if $mf(s)<1$, then $\mathbf{M}(s)$  is element-wise finite since
\be\label{mysli}
M(s)={1+f(s)\over1-mf(s)}.
\ee 

According to \eqref{MHG} we have   for all $n\ge1$ 
\begin{align*}
\mathbf M(s)=\mathbf H(s)+\sum_{k=1}^{n-1} (m\bar{\mathbf H}(s)\mathbf{G})^k\mathbf H(s)+(m\bar{\mathbf H}(s)\mathbf G)^n\mathbf M(s),
\end{align*}
and moreover
\begin{align*}
(\bar{\mathbf H}(s)\mathbf G)^n&=\left(\sum_{i=1}^\infty s^i \mathbf{H}^i\mathbf{1}^{\rm t}\mathbf{g}\right)^n\\
&=\left(\sum_{i=1}^\infty s^i \mathbf{H}^i\mathbf{1}^{\rm t}\right)\left(\sum_{i=1}^\infty s^i \mathbf{g}\mathbf{H}^i\mathbf{1}^{\rm t}\right)\ldots\left(\sum_{i=1}^\infty s^i \mathbf{g}\mathbf{H}^i\mathbf{1}^{\rm t}\right)\mathbf{g}\\
&=f^{n-1}(s)\sum_{i=1}^\infty s^i \mathbf{H}^i\mathbf{1}^{\rm t}\mathbf{g}=f^{n-1}(s)\bar{\mathbf H}(s)\mathbf G.
\end{align*}
It follows that for $s$ such that  $mf(s)<1$, the term 
$$(m\bar{\mathbf H}(s)\mathbf G)^n\mathbf M(s)=m^nf^{n-1}(s)\bar{\mathbf H}(s)\mathbf G\mathbf M(s)$$
vanishes as $n\to\infty$ and the previous two relations yield
\begin{align*}
\mathbf M(s)&=\mathbf H(s)+\sum_{n=1}^\infty (m\bar{\mathbf H}(s)\mathbf{G})^n\mathbf H(s)\\
&=\mathbf H(s)+\sum_{n=1}^\infty m^nf^{n-1}(s)\bar{\mathbf H}(s)\mathbf G\mathbf H(s)\\
&=\mathbf{H}(s)+{m\over1-mf(s)}(\mathbf{H}(s)-\mathbf{I})\mathbf{1}^{\rm t}\mathbf{g}\mathbf H(s).
\end{align*}
To finish the proof of Lemma \ref{mps} it remains to observe that
\begin{align*}
\mathbf M(s)&\ge\mathbf H(s)+\sum_{n=1}^\infty (m\bar{\mathbf H}(s)\mathbf{G})^n\mathbf H(s)\\
&=\mathbf H(s)+m\bar{\mathbf H}(s)\mathbf G\mathbf H(s)\sum_{n=0}^\infty (mf(s))^{n}
\end{align*}
implying that  $\mathbf{M}(s)$  is element-wise infinite for $s$ such that  $mf(s)\ge1$.
\end{proof}
\vspace{5mm}

The following well-known renewal theorem taken from Chapter XIII.4 in \cite{F1} will be used by us several times.

\begin{lemma} \label{Fe}
Let $A(s)=\sum_{n=0}^\infty a_ns^n$ be a probability generating function and $B(s)=\sum_{n=0}^\infty b_ns^n$ is a generating function for a non-negative sequence so that $A(1)=1$ while $B(1)\in(0,\infty)$. Then the  non-negative sequence defined by $\sum_{n=0}^\infty t_ns^n={B(s)\over1-A(s)}$ is  such that $t_n\to {B(1)\over A'(1)}$ as $n\to\infty$.
\end{lemma}

\subsection{Proof of Theorem \ref{the1}}\label{Spr}

We show first that if there are no phantom types and $\mathbf{H}$ has no zero rows, then for any $j$ there exists such $n=n_j$ that $m^{(n)}_{ij}>0$ for all $i$. This easily follows from the inequality
$\mathbf{M}^{n}\ge m\mathbf{H}\mathbf{1}^{\rm t}\mathbf g\mathbf{H}^{n-1}$ which comes from \eqref{MHG}.
Indeed, on one hand, all components of the vector $\mathbf{H}\mathbf{1}^{\rm t}$ are positive. On the other hand, the absence of phantom types implies that for the given $j$ we can find such $n=n_j\ge2$ that the $j$-th component of the  vector $\mathbf g\mathbf{H}^{n-1}$ is positive.

Assume from now on that $\mathbf M$ is irreducible and aperiodic. Statements  (i) and  (ii) of Theorem \ref{the1} follow directly from Lemma \ref{mps}.

 Assertion  (iii) is easily obtained by combining   \eqref{mysli} and Lemma \ref{Fe}. Using
 $$R^n\mathbf g\mathbf M^n\mathbf{1}^{\rm t}\to{1+f(R)\over mf'(R)}={1+m\over \beta m}$$
we conclude that  \eqref{nu} is possible if and only if   $\beta=\infty$. 

We show next that $\mathbf{H}(R)$ is element-wise finite provided $\alpha>-\infty$. First notice that $\mathbf g\mathbf{H}(R)\mathbf{1}^{\rm t}=f(R)=1/m$. On the other hand, in the absence of phantom types for any $i\ge1$ we can find a $k=k_i$ and a positive $c_i$ such that $\mathbf g\mathbf{H}^k\ge c_i\mathbf e_i$ implying
\[ c_i\mathbf e_i \mathbf{H}(R)\le \sum_{n=0}^\infty R^n\mathbf g\mathbf{H}^{k+n}\le R^{-k}\sum_{n=0}^\infty R^n\mathbf g\mathbf{H}^{n}=R^{-k}\mathbf g\mathbf{H}(R)\]
so that $\mathbf e_i \mathbf{H}(R)\mathbf{1}^{\rm t}\le c_i^{-1}R^{-k}m^{-1}<\infty$.

Now let $\beta<\infty$. Consider vectors $\mathbf u$ and $\mathbf v$ which, thanks to the just proved finiteness of $\mathbf{H}(R)$, are well-defined by \eqref{muu} and \eqref{muv}. The claimed equality $\mathbf v\mathbf u^{\rm t}=1$ follows from
\[
m\mathbf{g}\mathbf{H}(R)(\mathbf{H}(R)-\mathbf I)\mathbf{1}^{\rm t}=m\sum_{n=1}^\infty n R^n\mathbf{g}\mathbf H^n\mathbf{1}^{\rm t}=\beta,
\]
which is a consequence of
\begin{align*}
\mathbf{H}(s)\mathbf{H}(s)-\mathbf{H}(s)&=\sum_{i=1}^\infty \mathbf{H}^is^i\sum_{k=0}^\infty \mathbf{H}^ks^k\\
&=\sum_{i=1}^\infty \sum_{n=i}^\infty \mathbf{H}^ns^n=\sum_{n=1}^\infty \sum_{i=1}^n\mathbf{H}^ns^n=\sum_{n=1}^\infty n\mathbf{H}^ns^n.
\end{align*}
It remains to prove \eqref{mnu}. To this end  define a sequence of matrices $\mathbf B_n$ by 
\begin{align*}
\sum_{n=0}^\infty\mathbf B_ns^n&={m\over1-mf(Rs)}(\mathbf{H}(Rs)-\mathbf I)\mathbf{1}^{\rm t}\mathbf g\mathbf{H}(Rs),\quad s\in[0,1)
\end{align*}
so that $R^n\mathbf M^n=R^n\mathbf H^n+\mathbf B_n$ due to Lemma \ref{mps}.
According to Lemma \ref{Fe} we have an element-wise convergence
$\mathbf B_n\to  \mathbf u^{\rm t}\mathbf v$ as $n\to\infty$ and it remains to see that each element of $R^n\mathbf H^n$ converges to zero, since $\mathbf{H}(R)=\sum_{n\ge0}R^n\mathbf H^n$ is element-wise finite.

\subsection{Proof of Propositions \ref{p1}, \ref{p2}, and \ref{p3}}\label{SPPr}

\begin{proof} {\sc of Proposition \ref{p1}}. From \eqref{mn} and  \eqref{mysli} we obtain 
$$m^{(n)}\to m{1+f(1)\over1-mf(1)}=\tilde m$$ 
which together with  \eqref{nex} implies \eqref{epr} .
The statement on the convergence of the conditional distribution of $\mathbf{Z}^{(n)}$ follows from \eqref{cor}:
\[\mathbb{E}\big(\mathbf{s}^{\mathbf{Z}^{(n)}}|\mathbf{Z}^{(n)}\ne \mathbf{0},\mathbf{Z}^{(0)}=\mathbf{e}_i\big)\to{\sum_{j=1}^\infty \tilde h_{j}s_j\over1+\tilde m-\tilde m\sum_{j=1}^\infty \tilde g_js_j},\]
since $m^{(n)}\mathbf g^{(n)}\to m\mathbf g(\mathbf{I}-\mathbf{M})^{-1}=\tilde m\tilde{\mathbf g}$ and
\begin{align*}
\mathbf H^{(n)}&\sim\rho^n(\mathbf{u}^{\rm t}\mathbf{v}-(1-\mu)(1+m)^{-1}\mathbf{u}^{\rm t}\tilde m\tilde{\mathbf g})=\rho^n(1-\mu)(1+m)^{-1}\mathbf{u}^{\rm t}\tilde{\mathbf h}.
\end{align*}
\end{proof}
\begin{proof} {\sc of Proposition \ref{p2}}.
Lemma \ref{Fe} and relations  \eqref{mn}, \eqref{mysli} imply that in the critical case  
\[
m^{(n)}\sim n(1+m)\beta^{-1}.
\]
Thus the stated asymptotics for the survival probability follows from \eqref{nex}.
Using from \eqref{cor} we express the conditional moment generating function as
\begin{align*}
\mathbb E\left(e^{zn^{-1} \mathbf{Z}^{(n)}\mathbf{w}^{\rm t}}|\mathbf{Z}^{(n)}\ne \mathbf{0},\mathbf{Z}^{(0)}=\mathbf{e}_i\right)
={(1-h_{i0}^{(n)})^{-1}\sum_{j=1}^\infty h_{ij}^{(n)}e^{zw_j/n} \over 1+m^{(n)}-m^{(n)}\sum_{j=1}^\infty g_j^{(n)}e^{zw_j/n} }.
\end{align*}
Since $m^{(n)}\mathbf{g}^{(n)}\sim n(1+m)\beta^{-1}\mathbf v$, we obtain
\begin{align*}
(1-h_{i0}^{(n)})^{-1}\sum_{j=1}^\infty h_{ij}^{(n)}e^{zw_j/n} &\to1,\\
m^{(n)}\sum_{j=1}^\infty g_j^{(n)}(e^{zw_j/n}-1) &\to z(1+m)\beta^{-1}\mathbf{v}\mathbf{w}^{\rm t},
\end{align*}
and the asserted weak convergence follows from the convergence of moment generating functions
\begin{align*}
\mathbb E\left(e^{zn^{-1} \mathbf{Z}^{(n)}\mathbf{w}^{\rm t}}|\mathbf{Z}^{(n)}\ne \mathbf{0},\mathbf{Z}^{(0)}=\mathbf{e}_i\right)\to{1\over1-z(1+m)\beta^{-1}\mathbf{v}\mathbf{w}^{\rm t}}
\end{align*}
for all $z\in[0,z_0]$, where $z_0$ is some positive number (see \cite{C}).

\end{proof}
\begin{proof} {\sc of Proposition \ref{p3}}. Rewrite \eqref{mn} as 
$$R^{n-1}m^{(n)}=m\sum_{k=0}^{n-1}R^kR^{n-1-k}\mathbf g\mathbf M^k\mathbf{1}^{\rm t}$$
to obtain the following consequence of \eqref{mysli}
\begin{align*}
\sum_{n=1}^\infty (Rs)^{n-1}m^{(n)}={m(1+f(sR))\over (1-mf(sR))(1-Rs)}.
\end{align*}
Thus Lemma \ref{Fe} entails
\[
m^{(n)}\sim \rho^n(1+m)\beta^{-1}(\rho-1)^{-1}.
\]
This together with \eqref{mnu} and  \eqref{nex} gives the stated formula for the survival probability. The assertion on weak convergence is proved in a similar way as in the critical case above.
\end{proof}

\section*{Acknowledgements}

Several insightful comments and constructive suggestions of two anonymous referees have helped the author to significantly improve the presentation of the paper making the assertions more precise and the proofs more rigorous. This work was supported by the Swedish Research Council grant 621-2010-5623.

\end{document}